\documentclass{amsart}
\usepackage{calc,amssymb,amsthm,amsmath}
\RequirePackage[dvipsnames,usenames]{xcolor}
\usepackage{hyperref}
\hypersetup{
bookmarks,
bookmarksdepth=3,
bookmarksopen,
bookmarksnumbered,
pdfstartview=FitH,
colorlinks,backref,hyperindex,
linkcolor=Sepia,
anchorcolor=BurntOrange,
citecolor=MidnightBlue,
citecolor=OliveGreen,
filecolor=BlueViolet,
menucolor=Yellow,
urlcolor=OliveGreen
}
\usepackage{alltt}
\usepackage{multicol}
\usepackage{etex}
\usepackage{xspace}
\usepackage{stmaryrd}
\usepackage{tikz}

\usepackage[left=1.2in,top=1.0in,right=1.2in,bottom=1.0in]{geometry}

\usepackage{verbatim}
\usepackage{enumerate}

\usepackage{marginnote}

\newtheorem{theorem}{Theorem}[section]
\newtheorem{lemma}[theorem]{Lemma}
\newtheorem{proposition}[theorem]{Proposition}
\newtheorem{corollary}[theorem]{Corollary}

\theoremstyle{definition}

\theoremstyle{remark}
\newtheorem{remark}[theorem]{Remark}

\newtheorem{question}[theorem]{Question}
\numberwithin{equation}{subsection}

\DeclareMathOperator{\fm}{\mathfrak{m}}

\DeclareMathOperator{\fp}{\mathfrak{p}}
\usepackage{graphicx}

\usepackage[all,cmtip]{xy}

\DeclareMathOperator{\Tor}{Tor}

\DeclareMathOperator{\cC}{\mathcal{C}}
\DeclareMathOperator{\cF}{\mathcal{F}}
\DeclareMathOperator{\fa}{\mathfrak{a}}

\DeclareMathOperator{\Supp}{Supp}
\DeclareMathOperator{\reg}{reg}
\DeclareMathOperator{\Tot}{Tot}
\DeclareMathOperator{\Spec}{Spec}
\DeclareMathOperator{\Sing}{Sing}
\DeclareMathOperator{\ZZ}{\mathbb{Z}}
\DeclareMathOperator{\QQ}{\mathbb{Q}}

\renewcommand{\phi}{\varphi}

\renewcommand{\theta}{\vartheta}

\renewcommand{\epsilon}{\varepsilon}

\renewcommand{\to}[1][]{\xrightarrow{\ #1\ }}

\title{A note on the growth of regularity with respect to Frobenius}

\author{Wenliang Zhang}
\address{Department of Mathematics, Statistics, and Computer Science, University of Illinois, Chicago, IL 60607, USA}
\email{wlzhang@uic.edu}
\thanks{The author is partially supported by the National Science Foundation through grant DMS \#1405602.}

\begin{document}
\maketitle

\begin{abstract}
Let $R=k[x_1,\dots,x_n]/I$ be a graded $k$-algebra where $k$ is a field of prime characteristic and let $J$ be a homogeneous ideal in $R$. Denote $(x_1,\dots,x_n)$ by $\fm$. We prove that there is a constant $C$ (independent of $e$) such that the regularity of $H^s_{\fm}(R/J^{[p^e]})$ is bounded above by $Cp^e$ for all $e\geq 1$ and all integers $s$ such that $s+1$ is at least the dimension of the locus where $R/J$ doesn't have finite projective dimension. 
\end{abstract}

\section{introduction}
Let $S=k[x_1,\dots,x_n]$ be a polynomial ring over a field $k$ of characteristic $p>0$ with the standard grading. Let $\fm$ be the irrelevant maximal homogeneous ideal of $S$. For each graded artinian $S$-module $N$, we set $a(N):=\max\{j\mid N_j\neq 0\}$. And the Castelnuovo-Mumford regularity of a finitely generated graded $S$-module $M$ is defined as
\[\reg(M):=\max\{a(H^j_{\fm}(M))+j\mid j\in \ZZ\}.\]

For each ideal $\fa$ of $S$, let $\fa^{[p^e]}$ denotes the ideal generated by $\{r^{p^e}\mid r\in \fa\}$. \cite{KatzmanComplexityFrobenius} asked the following question:
\begin{question}
\label{question: linear growth wrt Frobenius}
Let $I,J$ be graded ideals of $S$. Does there exist a constant $C$ such that $\reg(\frac{S}{I+J^{[p^e]}})\leq Cp^e$ for all $e\geq 1$?
\end{question}

When such a constant exists, we say that $\reg(\frac{S}{I+J^{[p^e]}})$ grows linearly with respect to $p^e$ (or with respect to Frobenius). While the original motivation behind Question \ref{question: linear growth wrt Frobenius} concerns whether tight closure commutes with localization at a single element, it was discovered in \cite{KatzmanZhangRegularityandFJumpingNumbers} that a positive answer to Question \ref{question: linear growth wrt Frobenius} will imply that the $F$-jumping coefficients of any ideal in a finitely generated $k$-algebra are discrete. A positive answer to Question \ref{question: linear growth wrt Frobenius} is known in some cases; for instance, when $\dim(\Sing(S/I)\cap V(J))\leq 1$ (\cite{ChardinRegularityFunctors},~\cite{BrennerLinearBound},~\cite{KatzmanZhangRegularityandFJumpingNumbers}) and when $S/I$ has finite graded Frobenius representation type (\cite{KatzmanSchwedeSinghZhangRIngsofFrobeniusOperators}).

Let $R=k[x_1,\dots,x_n]/I$ be a standard graded ring over a field $k$ of characteristic $p>0$ and $\fm=(x_1,\dots,x_n)$. Let $F:R\to R$ denote the Frobenius endomorphism and $F^e_*R$ denote the target of $F^e$. For each $e\geq 1$, the $R$-module $F^e_*R$ is graded via 
\[\deg(F^e_*r)=\frac{1}{p^e}\deg(r)\]
for every homogeneous element $r\in R$. Consequently, for each graded $R$-module $M$, the $R$-module $M\otimes_R F^e_*R$ is graded by $\deg(z\otimes F^e_*r)=\deg(z)+\frac{1}{p^e}\deg(r)$.

Our main result is the following theorem:
\begin{theorem}
\label{main theorem}
Let $R=k[x_1,\dots,x_n]/I$ be a standard graded ring over a field $k$ of characteristic $p>0$ and $M$ be a finitely generated graded $R$=module. Set 
\[Npd(M)=\{\fp\in \Spec(R)\mid M_{\fp} {\rm \ does\ not\ have\ finite\ projective\ dimension\ over\ } R_{\fp}\}.\]
Then there is a constant $C$ such that 
\[a(H^s_{\fm}(M\otimes_RF^e_*R))\leq C\] 
for all $s\geq \dim(Npd(M))-1$ and all $e\geq 1$.
\end{theorem}

An immediate corollary is:
\begin{corollary}
\label{corollary: consequence for ideals}
Let $R=k[x_1,\dots,x_n]/I$ be a standard graded ring over a field $k$ of characteristic $p>0$ and $J$ be a homogeneous ideal of $R$. Then $a(H^s_{\fm}(R/J^{[p^e]}))$ grows linearly with respect to $p^e$ for each $s\geq \dim(Npd(R/J))-1$.

In particular, $a(H^s_{\fm}(R/J^{[p^e]}))$ grows linearly with respect to $p^e$ for $s\geq \dim(R/J)-1$.
\end{corollary}

\section*{acknowledgement}We would like to thank Mordechai Katzman for inspiring discussions.

\section{Proof of Theorem \ref{main theorem}}

\begin{remark}
\label{remark: Frobenius of localization}
It is well known that there is a degree-preserving isomorphism $R_f\otimes_RF^e_*R\xrightarrow{a\otimes r'\mapsto a^{p^e}r'} F^e_*R_f$ for each homogeneous element $f\in R$. Consequently, we have
\[H^j_{\fm}(F^e_*R)\cong F^e_*H^j_{\fm}(R).\] 
\end{remark}


Let $a(H^j_{\fm}(F^e_*R))$ denote $\max\{\rho\in \QQ\mid H^j_{\fm}(F^e_*R)_{\rho}\neq 0\}$. It follows immediately from Remark \ref{remark: Frobenius of localization} that
\begin{proposition}
\label{proposition: linear growth of Frobenius of R}
$a(H^j_{\fm}(F^e_*R))=\frac{1}{p^e}a(H^j_{\fm}(R))$.
\end{proposition}

\begin{lemma}
\label{lemma: linear growth tor of lc and module}
For each finitely generated graded $R$-module $M$, there is a constant $C_{ij}$ (independent of $e$) such that $a(\Tor^i_R(H^j_{\fm}(F^e_*R),M))\leq C_{ij}$ for all $e$. 
\end{lemma}
\begin{proof}
Let $\cF_{\bullet}$ be a graded free resolution of $M$. Write $\cF_i=\oplus_{n_i}R(\alpha_{n_i})$. Then 
\[\cF_i\otimes_R H^j_{\fm}(F^e_*R)\cong \oplus_{n_i} H^j_{\fm}(F^e_*R)(\alpha_{n_i}).\] 
The conclusion now follows from Proposition \ref{proposition: linear growth of Frobenius of R}.
\end{proof}

Consider the complex $\cC_{\bullet}$ which is the \v{C}ech complex associated with $x_1,\dots,x_n$ with the indices reversed as follows:
\[0 \leftarrow \cC_0 \leftarrow \cC_1\leftarrow \cdots \leftarrow \cC_n\leftarrow 0\]
where $\cC_j=\check{C}_{n-j}(\underline{x};R)$. Hence $H^j(\cC_{\bullet}\otimes_R N)=H^{n-j}_{\fm}(N)$ for each $R$-module $N$.

Let $M$ be a finitely generated graded $R$-module and $\cF_{\bullet}$ be a graded free resolution of $M$. Consider the double complex $\cC_{\bullet,\bullet}$:
\[
\xymatrix{
 & \vdots \ar[d] & \vdots \ar[d] &\\
0 & \cF_2\otimes_R\cC_0\otimes_RF^e_*R \ar[l] \ar[d]&  \cF_2\otimes_R\cC_1\otimes_RF^e_*R  \ar[l]\ar[d] & \cdots \ar[l]\\
0 & \cF_1\otimes_R\cC_0\otimes_RF^e_*R \ar[l] \ar[d]& \cF_1\otimes_R\cC_1\otimes_RF^e_*R \ar[l] \ar[d] &\cdots \ar[l]\\
0 & \cF_0\otimes_R\cC_0\otimes_RF^e_*R \ar[l] \ar[d]& \cF_0\otimes_R\cC_1\otimes_RF^e_*R \ar[l] \ar[d] &\cdots \ar[l]\\
 & 0 & 0
}
\]

Note that $\cC_{\bullet,\bullet}$ depends on $e$.

\begin{lemma}
\label{lemma: linear growth of total complex}
For each integer $\alpha$, there is a constant $C_{\alpha}$ (independent of $e$) such 
\[a(H_{\alpha}(\Tot(\cC_{\bullet,\bullet})))\leq C_{\alpha}.\]
\end{lemma}
\begin{proof}
It is clear that $\cC_{\bullet,\bullet}$ is a first quadrant double complex and hence it induces two spectral sequences both of which converge to the homology of the total complex (\cite[\S~5.6]{WeibelHomologicalAlgebraBook}). One of these spectral sequence comes from taking homology horizontally first and then vertically, whose $E^2$-page is given by $\Tor_i^R(H^j_{\fm}(F^e_*R),M)$. Since $H^j_{\fm}(F^e_*R)=0$ for all $j\geq d+1$ with $d=\dim(R)$, the $E^{\infty}$-page is the same as the $E_{d+1}$-page. Hence there are only finitely many $\Tor_i^R(H^j_{\fm}(F^e_*R),M)$ contributing to $H_{\alpha}(\Tot(\cC_{\bullet,\bullet}))$ for each integer $\alpha$. Lemma \ref{lemma: linear growth tor of lc and module} finishes the proof.
\end{proof}

To prove our main result, we need the following theorem due to Peskine-Szpiro \cite{PeskineSzpiro}:
\begin{theorem}[Peskine-Szpiro Acyclicity Lemma]
\label{theorem: acyclicity lemma}
Let $R$ be a noetherian commutative ring of characteristic $p>0$ and $N$ be a finite $R$-module. If $N$ has finite projective dimension, then
\[\Tor_i^R(F^e_*R,N)=0\]
for all $i\geq 1$ and all $e$.
\end{theorem}

Recall that $Npd(M)$ denote the locus where $M$ does not have finite projective dimension, {\it i.e.} \[Npd(M)=\{\fp\in \Spec(R)\mid M_{\fp} {\rm \ does\ not\ have\ finite\ projective\ dimension\ over\ } R_{\fp}\}.\]

\begin{proof}[Proof of Theorem \ref{main theorem}]
We consider the other spectral sequence induced by the double complex $\cC_{\bullet,\bullet}$, {\it i.e.} by taking homology vertically first and then horizontally:
\[E^2_{i,j}=H^{n-i}_{\fm}(\Tor_j^R(F^e_*R,M))\Rightarrow H_{i+j}(\Tot(\cC_{\bullet,\bullet})).\]
The differential coming into $H^s_{\fm}(M\otimes_RF^e_*R)=E^2_{n-s,0}$ is 0 since $E^2_{n-s+2,-1}=0$. Set $\dim(Npd(M))=t$. We can see that $\dim(\Tor_1(F^e_*R,M))\leq t$ by Theorem \ref{theorem: acyclicity lemma}. Hence, if $s\geq t-1$, we have $H^{s+2}_{\fm}(\Tor_1(F^e_*R, M))=0$ which implies that the differential going out of $H^s_{\fm}(M\otimes_RF^e_*R)=E^2_{n-s,0}$ is also 0. Therefore, if $s\geq t-1$, we have
\[H^s_{\fm}(M\otimes_RF^e_*R)=E^{\infty}_{n-s,0}.\]
According to Lemma \ref{lemma: linear growth of total complex}, there is a constant $C_{n-s}$ such that $a(H_{n-s}(\Tot(\cC_{\bullet,\bullet})))\leq C_{n-s}$. Since $E^{\infty}_{n-s,0}$ is a subquotient of $H_{n-s}(\Tot(\cC_{\bullet,\bullet}))$, one can see that 
\[a(H^s_{\fm}(M\otimes_RF^e_*R))\leq C_{n-s}.\]
This finishes the proof.
\end{proof}

\begin{proof}[Proof of Corollary \ref{corollary: consequence for ideals}]
It is well known that there is a degree-preserving isomorphism of $F_*^eR$-modules (hence $R$-modules):
\[R/J\otimes_RF^e_*R \xrightarrow{\bar{r}\otimes r'\mapsto \overline{r^{p^e}r'}}F^e_*R/J^{[p^e]}\]
for each graded ideal $J$ of $R$, where the grading on $F^e_*M$ is given by $\deg(F^e_*z)=\frac{1}{p^e}\deg(z)$ for any graded $R$-module $M$. It follows from Theorem \ref{main theorem} that there is a constant $C$ such that 
\[a(H^s_{\fm}(R/J\otimes_RF^e_*R))\leq C\]
for each $s\geq \dim(Npd(R/J))-1$ and all $e\geq 1$. Therefore,
\[a(H^s_{\fm}(R/J^{[p^e]}))=p^ea(F^e_*H^s_{\fm}(R/J^{[p^e]}))=p^ea(H^s_{\fm}(F^e_*R/J^{[p^e]})=p^ea(H^s_{\fm}(R/J\otimes_RF^e_*R))\leq Cp^e\]
for each $s\geq \dim(Npd(R/J))-1$ and all $e\geq 1$.
\end{proof}

\begin{remark}
Since $Npd(R/J)\subset \Sing(R)\cap \Supp(R/J)$, Corollary \ref{corollary: consequence for ideals} recovers \cite[3.5]{KatzmanZhangRegularityandFJumpingNumbers}.
\end{remark}

\bibliographystyle{skalpha}
\bibliography{lineargrowth}

\end{document}